\documentclass{paper}
\usepackage{amsthm}
\usepackage{amsmath}
\usepackage{amsfonts}
\usepackage{amssymb}
\usepackage{euscript}
\usepackage{xcolor}

\usepackage[export]{adjustbox}

\usepackage[letterpaper, margin=1in]{geometry}

\newcommand{\T}{\mathbb{T}}
\newcommand{\Nm}{\mathrm{Nm}}
\newcommand{\N}{\mathbf{N}}

\newcommand{\NN}{\mathbb{N}}
\newcommand{\Q}{\mathbb{Q}}
\newcommand{\Z}{\mathbb{Z}}
\newcommand{\R}{\mathbb{R}}
\newcommand{\Rbig}{R_{\mathrm{big}}}
\newcommand{\scr}[1]{\EuScript{#1}}
\newcommand{\OO}{\scr{O}}

\author{Nick Ramsey\thanks{nramsey@depaul.edu}}
\institution{DePaul University}
\title{Some symbolic dynamics in real quadratic fields with applications to inhomogeneous minima}

\begin{document}
\maketitle
\hrule
\smalltableofcontents

\theoremstyle{plain}
\newtheorem{theorem}{Theorem}
\newtheorem{lemma}{Lemma}
\newtheorem{corollary}{Corollary}
\newtheorem{proposition}{Proposition}
\theoremstyle{definition}
\newtheorem{remark}{Remark}
\newtheorem{question}{Question}
\newtheorem{myexample}{Example}

\begin{abstract}
Let $K$ be a real quadratic field. We use a symbolic coding of the action of a fundamental unit on the torus $\T_K = (K\otimes_\Q\R)/\OO_K$ to study the family of subsets $X_t\subseteq\T_K$ of norm distance $\geq t$ from the origin. As an application, we prove that inhomogeneous spectrum of $K$ contains a dense set of elements of $K$, and conclude that all isolated inhomogeneous minima lie in $K$. 
\end{abstract}

\section{Introduction}\label{sec:intro}

Let $D>1$ be a square-free positive integer and let $K = \Q(\sqrt{D})$ be the associated real quadratic field with ring of integers $\OO_K$. Let  $\N:K\longrightarrow \Q$ denote the absolute norm $\N(a) = |\Nm_{K/\Q}(a)| = |a\overline{a}|$, where $a\mapsto\overline{a}$ is Galois conjugation, and recall that the ring $\OO_K$ is called \emph{norm-Euclidean} if for all $a\in K$ there exists $q\in\OO_K$ such that $\N(a-q)<1.$ The ring of integers $\OO_K$ embeds as a lattice in the two-dimensional real vector space $V_K = K\otimes_\Q\R$, and we denote the quotient torus by $\T_K = V_K/\OO_K$.  Galois conjugation extends linearly to $V_K$, and the absolute norm extends accordingly to an indefinite quadratic form on $V_K$ that we also denote by $\N$. The norm is not $\OO_K$-invariant, but the function defined by \[M(P) = \inf_{Q\in \OO_K}\N(P-Q)\] is, and descends to a function on the torus $\T_K$ which we also denote by $M$.  The function $M$ is upper-semicontinuous (\cite{bsd1}, Theorem F).

The \emph{Euclidean minimum} of $K$ is defined by $M_1(K) = \sup_{P\in K}M(P)$. In particular, $M_1(K)<1$ implies that $\OO_K$ is norm-Euclidean, while $M_1(K)>1$ implies that it is not.  The second Euclidean minimum is defined by \[M_2(K) = \sup_{\substack{P\in K \\ M(P)<M_1(K)}}M(P)\] and $M_1(K)$ is said to be \emph{isolated} if $M_2(K)<M_1(K)$. We may proceed in this fashion producing Euclidean minima $M_i(K)$ until we find a non-isolated one. Note that upper-semicontinuity ensures that each of these suprema is actually achieved by some collection of points on the torus. The points $P$ in the above suprema are constrained to $K$, but we may remove that restriction and define the \emph{inhomogeneous minimum} of $K$ by $M_1(\overline{K}) = \sup_{P\in \T_K}M(P)$ and proceed as above to define the inhomogeneous minima $M_i(\overline{K})$. The \emph{inhomogeneous spectrum} of $K$ is simply the image $M(\T_K)$, and the \emph{Euclidean spectrum} of $K$ is its subset $M(K)$. 

The inhomogeneous minima demonstrate a variety of behavior, in some cases producing an infinite sequence of isolated minima while in others we find that $M_2(\overline{K})$ already fails to be isolated - see \cite{lemmermeyer} for an overview of results. Barnes and Swinnerton-Dyer proved in \cite{bsd2} that $M_1(K) = M_1(\overline{K})$ and conjectured that $M_1(\overline{K})$ is always isolated and rational, and that $M_2(\overline{K})$ is taken at a point with coordinates in $K$.  Numerous computations by other authors (\emph{e.g.} \cite{davenport2}, \cite{davenport4}, \cite{inkeri}, \cite{godwin55}, \cite{godwin63}, \cite{varnavides48}) suggest further that all inhomogeneous minima lie in $K$.

Much is known about these minima and spectra in higher degree when the unit group furnishes more automorphisms. Cerri showed in \cite{cerri} that if $K$ is has unit rank at least two, then $M_1(K)$ is taken at rational point, and hence rational. He showed further that if such a $K$ is not CM, then $M_1(K)$ is attained and isolated, and the Euclidean and inhomogeneous spectra coincide and consist of a sequence of rational numbers converging to $0$. Building on Cerri's work, Shapira and Wang proved in \cite{shapira-wang} that if $K$ has unit rank at least three then $M_1(K)$ is isolated and attained.

Returning to real quadratic $K$, the following is our main result. 
\begin{theorem}\label{main}
	The set of $M(\T_K)\cap K$ is dense in the inhomogeneous spectrum $M(\T_K)$. 
\end{theorem}
This theorem is proven by introducing the intermediate collection of points \[K/\OO_K \subseteq \widetilde{K}/\OO_K \subseteq \T_K\] whose coordinates belong to $K$, establishing that the minima of such points lie in $K$, and finally proving that the associated spectrum $M(\widetilde{K})$ is dense in the inhomogeneous spectrum. In the isolated case, this establishes the following extension of the conjecture of Barnes and Swinnerton-Dyer above.
\begin{corollary}\label{cor2}
	If $M_i(\overline{K})$ is isolated, then it is taken at point with coefficients in $K$ and we have $M_i(\overline{K})\in K$.
\end{corollary}
The method of proof also demonstrates that $M_1(\overline{K}) = M_1(K)$ is rational if it is isolated, but this was known already to Barnes and Swinnerton-Dyer (\cite{bsd2}, Theorem M). 

\section{The dynamical systems $X_t$}

By Dirichlet's unit theorem, we have $\OO_K^\times = \pm \varepsilon^\Z$ for some fundamental unit $\varepsilon$ of infinite order. We will later fix an embedding of $K$ into $\R$ and assume that $\varepsilon$ is chosen so that $\varepsilon>1$. Multiplication by $\varepsilon$ is absolute norm-preserving and extends by linearity to an endomorphism $\phi$ of $V_K$ that is also absolute norm-preserving. Since $\phi$ preserves the lattice $\OO_K$, it descends to an endomorphism of the torus $\T_K$ with the property that $M(\phi(P)) = M(P)$ for all $P\in \T_K$. The eigenvalues of $\phi$ are the embeddings of $\varepsilon$ into $\R$ and hence not roots of unity, so $\phi$ is an ergodic transformation of $\T_K$. This dynamical system, and a symbolic coding of it obtained from a Markov partition of the torus, is our main resource. Note that the subset $K/\OO_K$, which coincides with the set of periodic points for $\phi$, is traditionally referred to as the \emph{rational points} since they have rational $(x,y)$ coordinates (see Section \ref{coords}). 

For $t>0$, the $\phi$-invariant set $X_t = \{P\in \T_K\ |\ M(P)\geq t\}$ is closed by upper semicontinuity. We can describe $X_t$ alternatively by first noting that the open set \[\scr{U}(t) = \bigcup_{Q\in\OO_K}\{P\in V_K\ |\ \N(P-Q)<t\}\] is translation-invariant and descends to an open subset of $\T_K$, and then observing that $X_t$ is its complement. The sets $X_t$ have Lebesgue measure zero for $t>0$ since they are proper, closed, and $\phi$-invariant. It is natural to ask how the Hausdorff dimension $\dim(X_t)$ varies with $t$. That $\dim(X_t)\to 2$ as $t\to 0$ is a simple consequence of Theorem 2.3 of \cite{bm}. We prove in Corollary \ref{leftcontinuity} that $\dim(X_t)$ is left-continuous everywhere. Right-continuity remains an open question. 

We illustrate in the case $K=\Q(\sqrt{5})$. Davenport computed the Euclidean minima for this field in \cite{davenport2} and \cite{davenport4}, finding the infinite decreasing sequence of minima $M_1 = 1/4$, and for $i\geq 1$, \[M_{i+1} = \frac{F_{6i-2} + F_{6i-4}}{4(F_{6i-1}+F_{6i-3}-2)}\] where $F_k$ denotes the $k$th Fibonacci number.  Each of these minima is obtained at a finite collection of elements of $K/\OO_K$, and we have $M_i\longrightarrow t_\infty = (-1+\sqrt{5})/8\approx .1545$.  A plot of $\dim(X_t)$ in this case is given below. The zero-dimensional region necessarily covers $t>t_\infty$, since the collection of points giving rise to the Euclidean minima is countable. We prove in \cite{ramsey} that $\dim(X_t)>0$ for all $t<t_\infty$, while $\dim(X_{t_\infty})=0$.  In particular, $\dim(X_t)$ is continuous at $t_\infty$. 
\begin{center}
\includegraphics[scale=.4]{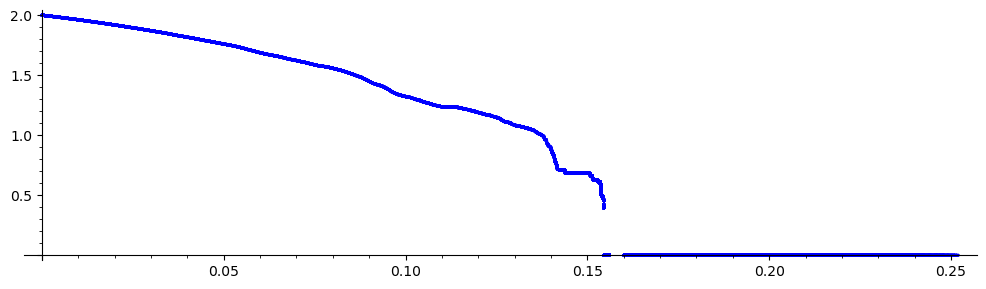}
\end{center}
The evident plateaus on this graph and its detail in Figure \ref{charondetail} have dynamical significance. The dimensions plotted here are actually upper bounds obtained by symbolically coding the torus dynamical system with a Markov partition and finding subshifts of finite type (SFTs) that contain the coding of $X_t$, as in Section \ref{sec:ub}.  As we will make precise in Proposition \ref{isosft}, a plateau will occur wherever it is possible to make such a bound tight and $X_t$ can be described directly by an SFT. The longest such plateau in the positive-dimensional region occurs around $t=.15$ (see Figure \ref{charondetail} for a detail), and we determine its endpoints and give an explicit symbolic coding of $X_t$ on this plateau in \cite{ramsey}.

\begin{figure}
\centering
\includegraphics[scale=.25]{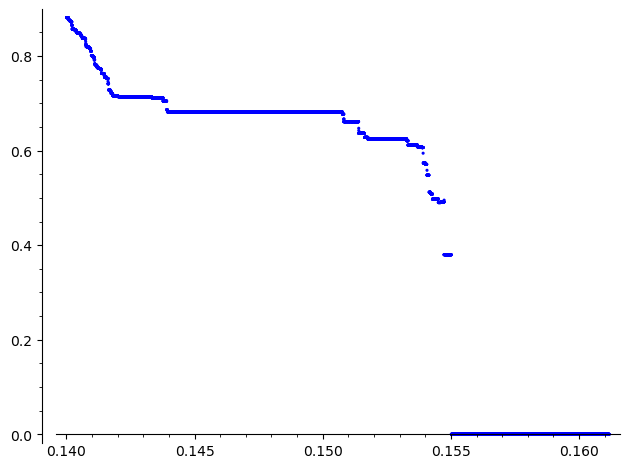}
\caption{Detail near $t=.15$}
\label{charondetail}
\end{figure}

\section{Coordinates and $K$-points}\label{coords}

Let us now take $K$ to be a subset of $\R$ by fixing an embedding, and take $\varepsilon$ to be a fundmental unit with $\varepsilon > 1$. Recall that $\{1, \alpha_K\}$ is a $\Z$-basis of $\OO_K$, where \[\alpha_K = \left\{\begin{array}{cc} \sqrt{D} & D\equiv 2,3\!\pmod{4} \\ \frac{1+\sqrt{D}}{2} & D\equiv 1\!\pmod{4}\end{array}\right.\] Coordinates with respect to this basis will be denoted $(x,y)$. The choice of embedding gives an isomoprhism
\begin{align*}
	 V_K = K\otimes_\Q\R \stackrel{\sim}{\longrightarrow} {} & \R\times\R \\
	 a\otimes 1  \longmapsto {} & (\overline{a},a)
\end{align*} 
of $\R$-algebras, and thus another coordinate system. Multiplication by $\varepsilon$ has the effect of multiplying by $\overline{\varepsilon} = \pm \varepsilon^{-1}$ in the first coordinate and $\varepsilon$ in the second coordinate. Accordingly, these are known as the \emph{stable} and \emph{unstable} coordinates and denoted $(s,u)$. Note that the absolute norm is simply $\N(s,u) = |su|$ in these coordinates, and that the coordinate transformations between $(x,y)$ and $(s,u)$ coordinates are $K$-linear. 

A point $P\in \T_K$ is called \emph{determinate} if it has a representative $Q\in V_K$ with $\N(Q) = M(P)$. It is shown in \cite{bm} (Theorem 2.6) that the set of determinate points is a meagre $F_\sigma$ set of measure zero and Hausdorff dimension 2. For a general point $P\in\T_K$, the following two lemmas help relate the value $M(P)$ to the more concrete values $\N(Q)$ for $Q\in V_K$. 

\begin{lemma}[\cite{bm}, Lemma 4.2]\label{bmlemma}
	Suppose that $P\in\T_K$ satisfies $M(P)<t$. There exists a point $Q=(s,u)\in V_K$ representing an element of the orbit of $P$ satisfying \[|s|, |u| < \sqrt{\varepsilon t}\] such that $\N(Q) = |su| <t$.
\end{lemma}

\begin{lemma}\label{orbclosure}
	Let $P\in \T_K$. There exists $Q\in V_K$ representing an element of the orbit closure of $P$ satisfying \[\N(Q) = M(Q) = M(P)\]
\end{lemma}
\begin{proof}
	Let $\Rbig$ denote the rectangle in $V_K$ given by $|s|,|u|< \sqrt{\varepsilon(M_1(K)+1)}$. By Lemma \ref{bmlemma}, there is for each $n\in\N$ a point $Q_n\in \Rbig$ representing an element of orbit of $P$ with 
	\begin{equation}\label{qnbd}
	\N(Q_n)<M(P)+\frac{1}{n}
	\end{equation}
	Since $\Rbig$ is bounded, there exists a subsequence $Q_{k_n}$ converging to some point $Q$.  Observe that \[M(Q) \leq \N(Q) = \lim_{k\to\infty}\N(Q_{k_n})\leq M(P)\] where the last inequality follows from (\ref{qnbd}). The definition of $Q$ ensures that it represents an element of the orbit closure of $P$. But this implies that $M(Q)\geq M(P)$ by upper-semicontinuity, since the value $M(P)$ is common to the entire orbit of $P$, and the result follows.
\end{proof}

Let us call a point in $V_K$ with rational $(x,y)$ coordinates a \emph{$\Q$-point}. Similarly a \emph{$K$-point} is one whose $(x,y)$ coordinates lie in $K$, or equivalently whose $(s,u)$ coordinates lie in $K$. The set of $\Q$-points coincides with $K/\OO_K$, which is also the set of periodic points for $\phi$. In particular, if $P$ is a $\Q$-point then the previous lemma immediately implies that $P$ is determinate and $M(P)\in\Q$. 
\begin{proposition}\label{KptMK}
	Let $P\in\T_K$ be a $K$-point. 
\begin{enumerate}
\item There exists $N\in\NN$ such that \[\phi^k(NP)\longrightarrow 0\ \ \ \ \mbox{as}\ \ \ \ |k|\to\infty\]
\item $M(P)\in K$
\end{enumerate}
\end{proposition}
\begin{proof}\
\begin{enumerate}
\item Since $P$ has $(s,u)$ coordinates in $K$, there exists $N\in\NN$ such that $NP$ has $(s,u)$ coordinates in $\OO_K$. In $(s,u)$ coordinates, the lattice $\OO_K\subseteq V_K$ is given by the set of pairs $(\overline{a},a)$ for $a\in\OO_K$. It follows by subtracting such elements that $NP$ has a representative whose stable coordinate vanishes, as well as a representative whose unstable coordinate vanishes. Now $\phi^{k}(NP)\to 0$ as $|k|\to \infty$ follows immediately. 
\item By the previous part, the orbit closure of the $K$-point $P$ consists of the orbit of $P$ together with a finite collection of $N$-torsion points on the torus. By Lemma \ref{orbclosure}, there exists $Q\in V_K$ representing an element of this orbit closure with $\N(Q) = M(P)$. Should $Q$ represent an $N$-torsion point, then $M(P)\in\Q$ since torsion points are $\Q$-points. On the other hand, if $Q=(s,u)$ represents an element of the orbit of $P$ then $Q$ is also a $K$-point, so we have $M(P)=\N(Q)=|su|\in K$. 
\end{enumerate}
\end{proof}

\section{Markov partitions}

For each $K$, the dynamical system $(\T_K,\phi)$ admits a Markov partition consisting of two open rectangles. Such a partition $\{R_0, R_1\}$ for $K=\Q(\sqrt{5})$ is pictured in Figure \ref{rq5morig} in $(x,y)$ coordinates. Figure \ref{morigsu} furnishes a uniform description  in $(s,u)$ coordinates of a two-rectangle Markov partition for any $K$. This description is simply the one provided by Adler in \cite{adler} translated into $(s,u)$ coordinates. See also \cite{snavely}, where the construction may originate. 
\begin{figure}[h]
    \centering
    \begin{minipage}{0.45\textwidth}
        \centering
        \includegraphics[width=0.5\textwidth]{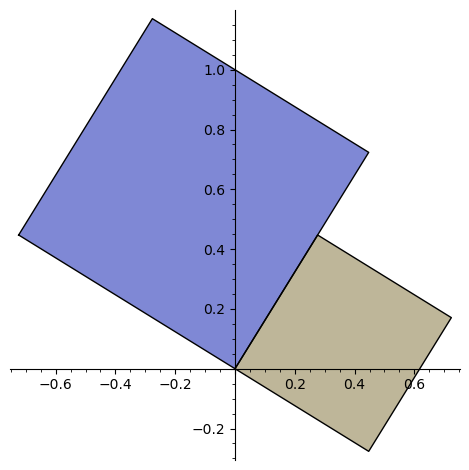}
	\caption{$\{R_0, R_1\}$ for $\Q(\sqrt{5})$}
	\label{rq5morig}
    \end{minipage}\hfill
    \begin{minipage}{0.45\textwidth}
        \centering
        \includegraphics[width=0.5\textwidth]{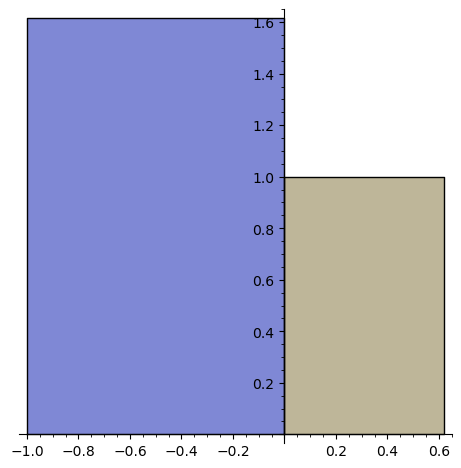}
	\put(-6,6){$\bullet$}
	\put(3,7){\small $(-\overline{\alpha_K},0)$}
	\put(-43,100){$\bullet$}
        \put(-37,100){\small $(0,\alpha_K)$}
        	\put(-43,65){$\bullet$}
        \put(-37,73){\small $(0,1)$}
        	\put(-102,6){$\bullet$}
        \put(-135,7){\small $(-1,0)$}
	\caption{\centering Original partition in $(s,u)$ coordinates (scale shown for $\Q(\sqrt{5})$)}
	\label{morigsu}
    \end{minipage}
\end{figure}

These two-rectangle partitions are typically not \emph{generators} essentially because the intersections $R\cap \phi(S)$ are generally disconnected. In the case of $\Q(\sqrt{5})$ however, the original partition $\scr{P}_0=\{R_0, R_1\}$ is a generator. Moreover, $R_0\cap\phi(R_0) = \emptyset$, while the remaining intersections consist of a single nonempty rectangle each. Let $\Sigma$ denote the subset of $\{0,1\}^\Z$ that avoids the string $00$ and let $\sigma:\Sigma\to\Sigma$ be the shift operator $\sigma(s)_i = s_{i+1}$.  The Markov generator property furnishes a map \[\pi: \Sigma \longrightarrow \T_K\] intertwining $\phi$ and the shift operator on $\Sigma$ that sends each string of coordinates to the unique point in $\T_K$ whose orbit has these coordinates: \[\pi(s) =  \bigcap_{n\in\NN}\overline{\bigcap_{i=-n}^n\phi^{-i}(s_i) }=   \bigcap_{i\in\Z} \phi^{-i}(\overline{s_i})\]  

\begin{remark}\label{closure}
	This construction ensures that $\phi^k\pi(s)\in\overline{s(k)}$ for all $k$. It follows that if the coordinate word of $A\in\scr{P}_n$ occurs in $s\in\Sigma$, then $\phi^k\pi(s)\in \overline{A}$ for a suitable $k\in\Z$. 
\end{remark}

The map $\pi$ is continuous, surjective, bounded-to-one, and essentially one-to-one.  Moreover, if $X\subseteq \T_K$ is a closed, invariant subset then $\pi$ restricts to a map \[\pi^{-1}(X)\longrightarrow X\] with the same properties, from which it follows that the entropy of $\phi|_X$ coincides with the entropy of the shift restricted to $\pi^{-1}(X)$. In the case of $X=X_t$, this entropy can be approximated by approximating the \emph{set} $\pi^{-1}(X_t)$ by subshifts obtained by refining the partition $\scr{P}_0$ and omitting some rectangles. The refinements are defined by taking $\scr{P}_n$ to consist of all nonempty intersections of the form 
\begin{equation}\label{tots}
\phi^{n}(A_{-n})\cap \cdots \cap \phi^{}(A_{-1})\cap A_0\cap\phi^{-1}(A_1)\cap\cdots \cap \phi^{-n}(A_n),\ \ \ A_i\in \scr{P}_0,
\end{equation}
and we say that this particular rectangle has \emph{coordinate word} $A_{-n}\cdots A_0\cdots A_n$. When a representative rectangle in the plane $V_K$ is needed for a member of $\scr{P}_n$, we take the one contained in the original footprint $R_0\cup R_1$.
 
The refinement $\scr{P}_n$ is also a Markov generator, and we have a refined coding $\pi_n:\Sigma_n\longrightarrow \T_K$ by the set of admissible strings in the alphabet $\scr{P}_n$. Note that $\Sigma_n$ is simply a ``block form" of $\Sigma$ and there is a canonical bijection $\Sigma\cong\Sigma_n$ compatible with the shift operator and the two codings of $\T_K$. While for general $K$, the partition $\{R_0, R_1\}$ is not a generator, in all cases the connected components of $A_0\cap \phi^{-1}(A_1)$ for $A_i\in \{R_0, R_1\}$ do comprise a Markov generator (see the proof of Theorem 8.4 of \cite{adler}).  Thus for any $K$ other than $\Q(\sqrt{5})$ we may let $\scr{P}_0$ denote this generator and then proceed as in the previous paragraph to produce refinements $\scr{P}_n$. In all cases, the diameter of $\scr{P}_n$ tends to zero as $n\to\infty$.

The following explicit construction of $\pi$ will be useful below. Here, $\scr{P}$ can be any Markov generator on $\T_K$ arising from a collection of rectangles in the plane $V_K$ with sides parallel to the stable and unstable axes. In particular we suppose we have a chosen representative in the plane for each member of $\scr{P}$, or equivalently a choice of stable and unstable interval of which this member is the product. Let $s\in \Sigma$, the set of all admissible bi-infinite strings in the alphabet $\scr{P}$.  First we show how to compute the unstable coordinate of $\pi(s)$. The intersections 
\begin{align*}
  r_0 = {} & s_0\\\ r_1 = {} &  s_0\cap \phi^{-1}(s_1)\\  r_2 = {} & s_0\cap \phi^{-1}(s_1)\cap \phi^{-2}(s_2) \\  \vdots &
\end{align*}
on the torus can be viewed in the plane as a sequence of rectangles within $s_0$ whose stable interval is constant (and equal to that of $s_0$) and whose unstable interval is shrinking. Up to similarity, the footprint of the unstable interval of $r_{i+1}$ inside that of $r_i$ depends only on the rectangles $s_i$ and $s_{i+1}$ and is independent of $i$. This is because $\phi$ simply scales by the positive number $\varepsilon$ in the unstable direction, preserving similarity. 

Given a rectangle in the plane with sides parallel to the stable and unstable axes, let us denote its stable and unstable intervals by $[\alpha_s(A),\beta_s(A)]$ and $[\alpha_u(A), \beta_u(A)]$, and let $\ell_*(A) = \beta_*(A)-\alpha_*(A)$ denote the corresponding lengths. For each pair $A,B\in \scr{P}$ with $AB$ admissible, we define \[\rho_u(A,B) = \frac{\alpha_u(A\cap\phi^{-1}(B))-\alpha_u(A)}{\ell_u(A)}\]  Pictured in the $(s,u)$ plane, this is the height of the bottom of the subrectangle $A\cap \phi^{-1}(B)$ inside $A$, expressed as a fraction of the total height of $A$, and is a measure of the footprint of this subrectangle in $A$ alluded to above. The left endpoint of the unstable interval of $r_i$ is then equal to \[\alpha_u(s_0)+\rho_u(s_0,s_1)\ell_u(s_0)+\rho_u(s_1,s_2)\frac{\ell_u(s_1)}{\varepsilon} + \cdots + \rho_u(s_{i-1},s_i)\frac{\ell_u(s_{i-1})}{\varepsilon^{i-1}},\]  so the unstable coordinate of $\pi(s)$ is given by the series
\begin{equation}\label{unstableseries}
\alpha_u(s_0)+\rho_u(s_0,s_1)\ell_u(s_0)+\rho_u(s_1,s_2)\frac{\ell_u(s_1)}{\varepsilon} + \rho_u(s_2,s_3)\frac{\ell_u(s_2)}{\varepsilon^2} +\cdots 
\end{equation}

The stable coordinate works the same way if $\overline{\varepsilon}>0$. Some additional care must be taken if $\overline{\varepsilon}<0$, since then $\phi$ is orientation-reversing in the stable direction and the footprints alternate with their mirror images up to similarity instead of being independent of $i$. In that case we define coefficients \[\rho_s^{+}(A,B) = \frac{\alpha_s(A\cap\phi(B))-\alpha_s(A)}{\ell_s(A)}\] and \[\rho_s^-(A,B) =  \frac{\beta_s(A)-\beta_s(A\cap\phi(B))}{\ell_s(A)},\] and the stable coordinate alternates between these: 
\begin{equation}\label{stableseries}
\alpha_s(s_0)+\rho^+_s(s_0,s_{-1})\ell_s(s_0)+\rho_s^-(s_{-1},s_{-2})\frac{\ell_s(s_{-1})}{\varepsilon} + \rho_s^+(s_{-2},s_{-3})\frac{\ell_s(s_{-2})}{\varepsilon^2}+\cdots 
\end{equation}

Let us now return to the partitions $\scr{P}_n$ derived from the two-rectangle partition above. If $s\in\Sigma$ is periodic, then the image $\pi(s)\in\T_K$ has periodic orbit, and hence is a $\Q$-point. The following lemma furnishes a similar description of some $K$-pts.
\begin{lemma}\label{periodK}
	Suppose that $s$ is eventually periodic in both directions. Then $\pi(s)$ is a $K$-point. 
\end{lemma}
\begin{proof}
	First observe that all members of our partitions $\scr{P}_n$ have coordinates in the field $K$. If $s$ is eventually periodic in both directions, then the series (\ref{unstableseries}) and (\ref{stableseries}) (and its analog in case $\overline{\varepsilon}>0$)  decompose into finitely many geometric series with all terms and coefficients expressible in terms of these coordinates, and the result follows.
\end{proof}

\begin{lemma}\label{newword}
	If $t'<t$ and $X_t\subsetneq X_{t'}$, then there exists a finite word occurring in $\pi^{-1}(X_{t'})$ that does not occur in $\pi^{-1}(X_t)$. 
\end{lemma}
\begin{proof}
	Suppose to the contrary that every word appearing in $\pi^{-1}(X_{t'})$ also occurs in $\pi^{-1}(X_t)$. We claim this forces $\pi^{-1}(X_{t'})$ to be contained in the closure of $\pi^{-1}(X_t)$, which is a contradiction since the latter is closed assumed distinct from the former.  Let $s\in \pi^{-1}(X_{t'})$, and for $k\in\NN$ let $w_k$ be the word $s(-k)\cdots s(0)\cdots s(k)$. By hypothesis, this word occurs in $\pi^{-1}(X_t)$, and by applying $\phi$ we may assume that it occurs centrally in some element $x_k\in \pi^{-1}(X_t)$.  In particular, $x_k$ and $s$ agree on the index interval $[-k,k]$, and it follows that $x_k\to s$ as $k\to \infty$, so $s$ lies in the closure of $\pi^{-1}(X_t)$. 
\end{proof}

\section{Upper bounds via trapping rectangles}\label{sec:ub}

Given a collection of rectangles $\scr{C}\subseteq \bigcup_n \scr{P}_n$, we denote by $\Sigma\langle \scr{C}\rangle$ the subshift of $\Sigma$ that avoids the coordinate words of elements of $\scr{C}$. If $\scr{C}$ is finite, then there is a largest $n$ for which $\scr{P}_n$ contains an element of $\scr{C}$. Now every element of $\scr{C}$ breaks up into rectangles in $\scr{P}_n$, and we let $\scr{C}'\subseteq \scr{P}_n$ denote the collection of rectangles occurring in this fashion. Under the identification $\Sigma\cong\Sigma_n$, the subshift $\Sigma\langle\scr{C}\rangle$ can alternately be described as the collection of $s\in \Sigma_n$ for which $s(k)\notin\scr{C}'$ for all $k\in\Z$. 

Let $I\subseteq\OO_K$ be a finite set of lattice points and let \[\scr{U}(t,I) = \bigcup_{Q\in I} \{P\in V_K\ |\ N(P-Q)<t\}\] and let
\[\scr{T}_n(t,I) = \left\{ A \in \scr{P}_n\ \left|\ \overline{A}\subseteq \scr{U}(t,I) \right.\right\}\] be the collection of rectangles in $\scr{P}_n$ whose closures are trapped within the norm-distance $t$ ``neighborhood'' of some lattice point in $I$. The following lemma says that $\Sigma\langle\scr{T}_n(t,I)\rangle$ is an upper bound not only for $X_t$ but for $X_{t-\eta}$ for some $\eta>0$. 
\begin{lemma}\label{upperbound}
There exists $\eta>0$ such that $\pi^{-1}(X_{t-\eta}) \subseteq \Sigma\langle\scr{T}_n(t,I)\rangle$.
\end{lemma}
\begin{proof}
	The elements of $\scr{T}_n(t,I)$ have closures contained in the $\scr{U}(t,I)$ and thus in $\scr{U}(t-\eta,I)$ for some $\eta>0$ since $I$ is finite. If $s\in \Sigma$ contains the coordinates of $A\in\scr{T}_n(t,I)$, then $\phi^k\pi(s)$ lies in $\overline{A}$ for some $k$, by Remark \ref{closure}. But then $M(\pi(s)) = M(\phi^k\pi(s)) < t-\eta$. Thus $s\notin \pi_n^{-1}(X_{t-\eta})$. 
\end{proof}
The entropy of $\phi$ on $X_t$ is thus bounded above by the shift entropy of  $\Sigma\langle\scr{T}_n(t,I)\rangle$, which is computable by Perron-Frobenius theory. These upper bounds depend on the set $I\subseteq \OO_K$ and improve as $I$ grows. The following proposition and its corollary ensure that it is possible to choose $I$ so that the bounds are tight in the limit as $n\to \infty$. 
\begin{proposition}\label{uppertight} 
	There exists a finite set $I_K$ such that if $I_K\subseteq I$ and $t'<t$, then for $n$ sufficiently large we have 
	\begin{equation}\label{tighteq}
	\pi^{-1}(X_t)\subseteq \Sigma\langle \scr{T}_n(t,I)\rangle\subseteq \pi^{-1}(X_{t'})
	\end{equation}
	 In particular, for such $I$ we have 
	 \[\pi^{-1}(X_t) = \bigcap_{n\geq 0}\Sigma\langle \scr{T}_n(t,I)\rangle\]
\end{proposition}
\begin{proof} 
The second assertion here follows immediately from the first. Let $\Rbig$ denote the rectangle in $V_K$ given by $|s|,|u|<\sqrt{\varepsilon(M_1(K)+1)}$ and let $I_K$ be the set of all $q\in\OO_K$ such that $R-q$ meets $\Rbig$ for some $R\in\scr{P}_0$. The set $I_K$ is finite and necessarily contains any $q$ for which there exists some $A\in \scr{P}_n$ such that $A-q$ meets $\Rbig$.  Since the diameter of $\scr{P}_n$ tends to zero, there exists $N\in\NN$ such that $n\geq N$ implies that every translate of $A$ that meets the region defined by $\N<t'$ in $\Rbig$ must have closure entirely contained within the region $\N<t$. 

The first containment in (\ref{tighteq}) is clear from the preceding lemma, and we prove the second by contrapositive. Suppose that $s\in\Sigma$ is not in $\pi^{-1}(X_{t'})$. Then with $P=\pi(s)$ we have $M(P) < t'$, so  \[M(P)< t'' = \min(t',M_1(K)+1)\] Thus we may take $Q=(s,u)$ as in Lemma \ref{bmlemma} representing an element of the orbit of $P$ with $\N(Q)<t''$ and \[|s|,|u| < \sqrt{\varepsilon t''}\] In particular, $Q\in \Rbig$. For each $n$, the point $Q$ lies in the $\OO_K$-translates of the the closures of one or more members of the partition $\scr{P}_n$. Let $A\in \scr{P}_n$ and $q\in \OO_K$ such that $Q\in \overline{A}-q$. Thus $\overline{A}-q$ meets the region defined by $\N<t'$ in $\Rbig$, which requires that $A-q$ meet this region since $A$ is open, and hence $q\in I_K\subseteq I$. Now if $n\geq N$, it follows that $A\in \scr{T}_n(t,I)$.

By Remark \ref{closure}, the $0$th symbolic coordinate of any element of $\pi_n^{-1}(Q)$ must be a member of $\scr{T}_n(t,I)$, which implies that each element of $\pi_n^{-1}(P)$ has some symbolic coordinate in $\scr{T}_n(t,I)$.  This is to say that each element of $\pi^{-1}(P)$, including $s$, contains the coordinates of some element of $\scr{T}_n(t,I)$, and thus $s\notin\Sigma\langle \scr{T}_n(t,I)\rangle$.
\end{proof}

\begin{corollary}
	If $I_K\subseteq I$, then \[h(\phi|X_t) = \lim_{n\to \infty} h(\sigma|\Sigma\langle \scr{T}_n(t,I)\rangle\]
\end{corollary}
\begin{proof}
	Let $\mu_n$ be a measure of maximal entropy for $\Sigma\langle\scr{T}_n(t,I)\rangle$. Extended to $\Sigma$, this sequence of measures has some weak-$*$ limit point $\mu$ in the convex, compact space of invariant probability measures on $\Sigma$. The measure $\mu$ is supported on the intersection $\pi^{-1}(X_t)$, and by upper semi-continuity of entropy in subshifts we have 
\begin{align*}
 h(\sigma|\pi^{-1}(X_t))\geq h_\mu(\sigma) \geq {} & \limsup h_{\mu_n}(\sigma) \\  = {} & \limsup h(\sigma|\Sigma\langle\scr{T}_n(t,I)\rangle \geq  \liminf h(\sigma|\Sigma\langle\scr{T}_n(t,I)\rangle \geq h(\sigma|\pi^{-1}(X_t))
\end{align*}	
This implies that $\mu$ is a measure of maximal entropy for $\pi^{-1}(X_t)$, as well as the claim. 
\end{proof}

\begin{corollary}\label{leftcontinuity}
	The function $t\longmapsto \dim(X_t)$ is left-continuous at each point. 
\end{corollary}
\begin{proof}
	The dimension of a closed, invariant subset $X\subseteq \T_K$ is related to the entropy of $\phi$ on $X$ via \[\dim(X) = \frac{2h(\phi|X)}{\log(\varepsilon)},\] so it suffices to prove that $t\longmapsto h(\phi|X_t)$ is left continuous. Since this function is decreasing, left-discontinuity at $t$ would imply there exists $B>0$ such that \[h(\phi|X_{t-\eta})-h(\phi|X_{t})\geq B\ \ \mbox{for all}\ \ \eta>0\] By the previous corollary we know there exists $n\in \mathbb{N}$ with \[h(\sigma|\Sigma\langle \scr{T}_n(t,I_K)\rangle) - h(\phi|X_{t}) < B\] Now Lemma \ref{upperbound} ensures that $\Sigma\langle \scr{T}_n(t,I_K)\rangle$ contains $\pi^{-1}(X_{t-\eta})$ for some $\eta>0$, which implies \[h(\sigma|\Sigma\langle \scr{T}_n(t,I_K)\rangle)\geq h(\sigma|\pi^{-1}(X_{t-\eta})) = h(\phi|X_{t-\eta}),\] contradicting the inequalities above. 
\end{proof}

\section{Applications to the inhomogeneous spectrum}\label{sec:apps}

The plot of $\dim(X_t)$ contains a number of plateaus as illustrated in the case $K=\Q(\sqrt{5})$ above. Sometimes these are actually set-theoretic plateaus, and the following proposition demonstrates that $\pi^{-1}(X_t)$ is particularly simple in such cases. 
\begin{proposition}\label{isosft}
	Suppose that $X_t = X_{t-\eta}$ for some $\eta>0$. Then $\pi^{-1}(X_t)$ is a subshift of finite type.
\end{proposition}
\begin{proof}
	By Proposition \ref{uppertight}, we may choose $n\in\NN$ so that \[\pi^{-1}(X_t)\subseteq \Sigma\langle\scr{T}_n(t,I)\rangle\subseteq \pi^{-1}(X_{t-\eta})=\pi^{-1}(X_t)\] Thus $\pi^{-1}(X_t) = \Sigma\langle\scr{T}_n(t,I)\rangle$, which is expressible directly as an SFT via a 0-1 matrix when viewed in block form in $\Sigma_{m}$ for some $m$ (namely, any $m\geq n-1$). 
\end{proof}

Finally, we prove the main density result.
\begin{proof}[Proof of Theorem \ref{main}]
	First suppose that $t\in M(\T_K)$ is an isolated point. By the previous proposition, $\pi^{-1}(X_t)$ is a subshift of finite type, which is to say that it can be described by a 0-1 transition matrix when viewed in block form $\Sigma_m$ for some $m$. Since $t$ is isolated, we know by Lemma \ref{newword} that $\pi^{-1}(X_t)$ contains a finite word $w$ that does not occur in $\pi^{-1}(X_{>t})$. Let $s=uwv\in\Sigma$ with $M(\pi(s))=t$. Viewed in $\Sigma_m$, there is by the Pigeanhole Principle a repeated block in both $u$ and $v$. We can then truncate $u$ and $v$ and loop the segment between these books indefinitely to produce an element $s'\in \pi^{-1}(X_t)$ that contains $w$ and is eventually periodic in both directions. Then $\pi(s')$ is a $K$-point by Lemma \ref{periodK}, and $M(\pi(s'))=t$ since $s'$ contains $w$.
	
Now suppose that $t\in M(\T_K)$ is not isolated, so there is a strictly monotone sequence $(t_k)$ in $M(\T_K)$ with $t_k\to t$. Fixing $k\in\NN$, we will show that there is a $K$-point $P$ with such that $M(P)$ lies between $t$ and $t_k$, which will finish the density claim. 
First suppose that $(t_k)$ increases to $t$. Since $t_{k+1}\in M(\T_K)$, Lemma \ref{newword} ensures that there exists $s\in\pi^{-1}(X_{t_{k+1}})$ containing a word $w$ that does not occur in $\pi^{-1}(X_t)$. Now take $n$ large enough so that \[\pi^{-1}(X_{t_{k+1}})\subseteq \Sigma\langle\scr{T}_n(t_{k+1},I)\rangle\subseteq \pi^{-1}(X_{t_k})\] as in Proposition \ref{uppertight}. Since $s$ belongs to the SFT $\Sigma\langle\scr{T}_n(t_{k+1},I)\rangle$, we can modify it by looping its ends as in the previous paragraph to obtain another element $s'$ of this SFT that also contains $w$. But then we have $t_k\leq M(\pi(s'))<t$, so $P=\pi(s')$ is the desired $K$-point. 

Now suppose that $(t_k)$ is decreasing. Since $t_{k+1}\in M(\T_K)$, Lemma \ref{newword} ensures there is word $w$ occurring in $\pi^{-1}(X_{t_{k+1}})$ that does not occur in $\pi^{-1}(X_{t_{k}})$. Now take $n$ large enough so that \[\pi^{-1}(X_{t_{k+1}})\subseteq \Sigma\langle\scr{T}_n(t_{k+1},I)\rangle\subseteq \pi^{-1}(X_{t})\] and proceed as before to produce $s'\in\Sigma\langle\scr{T}_n(t_{k+1},I)\rangle$ that contains $w$ and is eventually periodic in both directions. We have $t\leq M(\pi(s'))<t_k$, and again $P=\pi(s')$ is the desired $K$-point. 
\end{proof}

\bibliography{rqdyn}
\bibliographystyle{plain}

\end{document}